\documentclass[12pt,a4]{amsart}
\usepackage{amsthm,amsmath,amssymb}
\usepackage{titlesec}
\usepackage{mathrsfs}
\allowdisplaybreaks[1]
\makeatletter
\def\section{\@startsection{section}{1}%
  \z@{.7\linespacing\@plus\linespacing}{.5\linespacing}%
  {\normalfont\scshape\centering}}
\def\subsection{\@startsection{subsection}{2}%
  \z@{.5\linespacing\@plus.7\linespacing}{-.5em}%
  {\normalfont\bfseries}}
\makeatother
\titleformat*{\section}{\large\bfseries}
\titleformat*{\subsection}{\large\bfseries}
\newtheorem{theorem}{Theorem}[section]

\newtheorem{lemma}{Lemma}[section]

\theoremstyle{remark}
\newtheorem{rem}{Remark}

\title[The Discrete Mean of $Z^{(j)}(t)$]
{On the Discrete Mean of the higher Derivative of Hardy's $Z$-Function}
\author{Hirotaka Kobayashi}
\date{}
\address{Graduate School of Mathematics, Nagoya University, Furocho, Chikusaku, Nagoya 464-8602, Japan}
\email{m17011z@math.nagoya-u.ac.jp}
\begin{document}

\maketitle

\begin{abstract}
Y\i ld\i r\i m obtained an asymptotic formula of the discrete moment of $|\zeta(\frac{1}{2}+it)|$ over the zero of the higher derivatives of Hardy's $Z$-function.
We give a generalization of his result on Hardy's $Z$-function.
\end{abstract}

\section{Introduction}
Hardy's $Z$-function is defined as
\begin{equation}
Z(t)=\chi\left(\frac{1}{2}+it \right)^{-\frac{1}{2}}\zeta\left(\frac{1}{2}+it\right),
\end{equation}
where $\chi(s)=2^s\pi^{s-1}\sin(\pi s/2)\Gamma(1-s)$ which comes from the functional equation for $\zeta(s)$.
In this paper, we discuss the mean value of $Z^{(j)}(t)$ over the zeros of $Z^{(k)}(t)$,
where $Z^{(j)}(t)$ is the $j$-th derivative of Hardy's $Z$-function.
We denote the complex variable by $s=\sigma+it$ with $\sigma, t \in \mathbb{R}$.
Throughout this article, we assume that the Riemann Hypothesis (RH) is true.

In 1985, Conrey and Ghosh \cite{C&G} showed that 
\begin{equation*}
\sum_{0 < \gamma \leq T} \max_{\gamma < t \leq \gamma^{+}} \left|\zeta \left(\frac{1}{2}+it \right) \right|^2
\sim \frac{e^2-5}{4\pi}TL^2,
\end{equation*}
where $\gamma$ and $\gamma^{+}$ are ordinates of consecutive zeros of $\zeta(s)$ and $L=\log \frac{T}{2\pi}$.
It is known that the zeros of the derivative of Hardy's $Z$-function are interlaced with those of Hardy's $Z$-function.
Thus this summand means the extremal value of Hardy's $Z$-function.
They calculated the integral
\begin{equation*}
\frac{1}{2\pi i}\int_{C} \frac{Z_1'}{Z_1}(s)\zeta(s)\zeta(1-s)ds,
\end{equation*}
where $C$ is positively oriented rectangular path with vertices $c+i$, $c+iT$, $1-c+iT$ and $1-c+i$ where $c=\frac{5}{8}$, and $Z_1(s)$ is defined by
\begin{equation*}
Z_1(s) :=\zeta'(s)-\frac{1}{2}\omega(s)\zeta(s),
\end{equation*}
where
\begin{equation*}
\omega(s)=\frac{\chi'}{\chi}(s)=\log 2\pi+\frac{\Gamma'}{\Gamma}(s)-\frac{\pi}{2}\tan \frac{\pi s}{2}.
\end{equation*}
Indeed, we can see that $|Z_1(\frac{1}{2}+it)|=|Z'(t)|$.

On the other hand, Y\i ld\i r\i m considered a generalization.
He gave an asymptotic formula
\begin{equation*}
\sum_{\gamma_k\leq T}\left|\zeta \left(\frac{1}{2}+i \gamma_k\right) \right|^2\sim
\begin{cases}
\frac{TL^2}{2\pi}(1+\frac{1}{k}+O(\frac{\log k}{k^2})) & (k \ \text{odd and} \ k>1) 
\vspace{2mm}\\
\frac{TL^2}{2\pi}(1-\frac{3}{k}+O(\frac{\log k}{k^2})) & (k \ \text{even}),
\end{cases}
\end{equation*}
where $\gamma_k$ runs over the zeros of the $k$th derivative of Hardy's $Z$-function.
He consider the logarithmic derivative of $\mathscr{Z}_k(s,T)=\left(\frac{L}{2}+\frac{d}{ds} \right)^k\zeta(s)$.

In this article, we will prove the following theorem.
\begin{theorem}
Let $j$ and $k$ be fixed non-negative integers. If RH is true, then as $T\rightarrow \infty$,
\begin{align*}
&\quad \sum_{0<\gamma_k\leq T}\left|Z^{(j)}(\gamma_k) \right|^2 \\
&=\delta_{0,k}\frac{T}{2^{2j+1}(2j+1)\pi}\left(\log \frac{T}{2\pi}\right)^{2j+2} \\
&-\frac{(k+1)\{1+(-1)^j \}}{2^{2j+1}(j+1)^2}\frac{T}{2\pi}\left(\log \frac{T}{2\pi} \right)^{2j+2} \\
&\quad +\sum_{u=1}^{j}\frac{1}{2j+1-u}\frac{j!}{(j-u)!}(-1)^{-u}\sum_{g=1}^{k}\frac{1}{\theta_g^{u+1}}\frac{T}{2^{2j+1}\pi}\left(\log \frac{T}{2\pi} \right)^{2j+2} \\
&\quad +(-1)^{j+1}\sum_{g=1}^{k}\frac{(j!)^2}{\theta_g^{2j+2}}\frac{T}{2^{2j+2}\pi}\left(\log \frac{T}{2\pi} \right)^{2j+2} \\
&\quad +(-1)^{j}(j!)^2\sum_{g=1}^{k}\frac{\left(\frac{T}{2\pi}\right)^{z_g-1}}{\theta_g^{2j+2}} \left(\sum_{\mu=0}^{j}\frac{\theta_g^{\mu}}{\mu !} \right)^2\frac{T}{2^{2j+2}\pi}\left(\log \frac{T}{2\pi} \right)^{2j+2} \\
&\quad+O_{j,k}\left(T(\log T)^{2j+1}\right),
\end{align*}
where $\delta_{0,k}$ is Kronecker's delta, $z_g \ (g=1,2,\cdots,k)$ are the zeros of $\mathscr{Z}_k(s,T)$ with $z_g=1-\frac{2}{L}\theta_g+O(\frac{1}{L^2})$, and $\theta_g$ satisfies $\sum_{\mu=0}^{k}\frac{\theta_g^{\mu}}{\mu !}=0$.
When $j=0$ or $k=0$, we consider the sums on the right-hand side as the empty sums.
\end{theorem}

At the last main term, since $L=\log \frac{T}{2\pi}$, we see that
\begin{equation*}
\left(\frac{T}{2\pi} \right)^{z_g-1}=e^{-2\theta_g+O(\frac{1}{L})}.
\end{equation*}
Therefore we can write the approximate formula in the form $C_{j,k}TL^{2j+2}$.

\begin{rem}
Matsuoka \cite{M} proved that the zeros of $Z^{(k+1)}(t)$ are interlaced with those of $Z^{(k)}(t)$ for sufficiently large $t$.
Therefore our sum contains the mean square of the extremal value of $|Z^{(j)}(t)|$.
\end{rem}

\begin{rem}
When $k=2$, it is not clear whether the coefficient of Y\i ld\i r\i m's asymptotic formula is positive or negative, hence his result does not give precise information, and our main theorem too.
This is because we have no exact information on the location of zeros of $\mathscr{Z}_2(s,T)$ near $s=1$.
In general, it is difficult to confirm even $C_{j,k}\geq 0$.
However, we can verify $C_{k,k}=0$ because it is known that
\begin{equation*}
\sum_{g=1}^{k}\frac{1}{\theta_g^{2k+2}}=\frac{(-1)^{k+1}+1}{k!(k+1)!}
\end{equation*}
and
\begin{equation*}
\sum_{g=1}^{k}\frac{1}{\theta_g^u}=
\begin{cases}
-1 & (u=1), \\
0  & (2\leq u\leq k), \\
\frac{1}{k!} & (u=k+1)
\end{cases}
\end{equation*}
(see \cite{Y-0}).
\end{rem}

In the proof, we apply a continuous mean value which asserts, for each $j=0, 1, 2, \cdots$, and any sufficiently large $T$,
\begin{equation}
\int_{0}^{T} Z^{(j)}(t)^2dt
=\frac{1}{4^j(2j+1)}TP_{2j+1}\left(\log \frac{T}{2\pi}\right)+error,
\end{equation}
where $P_{2j+1}(x)$ is the monic polynomial of degree $2j+1$ given by
\begin{equation*}
P_{2j+1}(x)=W_{2j+1}(x)+(4j+2)\sum_{n=0}^{2j} \binom {2j}{n}(-2)^n c_n W_{2j-n}(x),
\end{equation*}
in which
\begin{equation}\label{c_h}
W_{g}(v)=\frac{1}{e^v}\int_{0}^{e^v}(\log u)^g du,
\quad \zeta(s)=\frac{1}{s-1}+ \sum_{n=0}^{\infty}\frac{(-1)^n c_n}{n!}(s-1)^n.
\end{equation}
The $c_n$ are called the Stieltjes constants.
It is Hall \cite{Hall} who showed this result first with the error $\ll T^{\frac{3}{4}}(\log T)^{2j+\frac{1}{2}}$.
Later, Minamide and Tanigawa \cite{Mi-T} improved the error to $O(T^{\frac{1}{2}}(\log T)^{2j+1})$.
For our proof, Hall's error is not sufficient and we need the estimate of Minamide and Tanigawa.

\section{Preliminaries}

We introduce a meromorphic function associated with $k$-th derivatives of Hardy's $Z$-function. Let $Z_0(s)=\zeta(s)$, and for $k\geq 1$, we define $Z_k(s)$ as
\begin{equation}
Z_k(s)=Z_{k-1}'(s)-\frac{1}{2}\omega(s)Z_{k-1}(s).
\end{equation}

Furthermore let $f_0(s)=1$, and we define $f_k(s)$, for $k\geq 1$, as
\begin{equation*}
f_k(s)=f'_{k-1}(s)-\frac{1}{2}\omega(s)f_{k-1}(s).
\end{equation*}
Then we can see that
\begin{equation}\label{Z_k-binom}
Z_k(s)=\sum_{\mu=0}^{k}\binom{k}{\mu}f_{k-\mu}(s)\zeta^{(\mu)}(s)
\end{equation}
by induction on $k$. Actually,

\begin{align*}
&\quad
Z_{k+1}(s) \\
&=
Z_k'(s)-\frac{1}{2}\omega(s)Z_k(s) \\
&=
\sum_{\mu=0}^{k}\binom{k}{\mu}f_{k-\mu}'(s)\zeta^{(\mu)}(s)+\sum_{\mu=0}^{k}\binom{k}{\mu}f_{k-\mu}(s)\zeta^{(\mu+1)}(s) \\
&\quad
-\frac{1}{2}\omega(s)\sum_{\mu=0}^{k}\binom{k}{\mu}f_{k-\mu}(s)\zeta^{(\mu)}(s) \\
&=
\sum_{\mu=0}^{k}\binom{k}{\mu}f_{k+1-\mu}(s)\zeta^{(\mu)}(s)+\sum_{\mu=0}^{k}\binom{k}{\mu}f_{k-\mu}(s)\zeta^{(\mu+1)}(s) \\
&=
f_{k+1}(s)\zeta(s)+\sum_{\mu=1}^{k}\left \{ \binom{k}{\mu}+\binom{k}{\mu-1} \right \}f_{k+1-\mu}(s)\zeta^{(\mu)}(s)+\zeta^{(k+1)}(s) \\
&=
\sum_{\mu=0}^{k+1}\binom{k+1}{\mu}f_{k+1-\mu}(s)\zeta^{(\mu)}(s).
\end{align*}
To show the last equality, we use

\begin{equation*}
\binom{k}{\mu}+\binom{k}{\mu-1}=\binom{k+1}{\mu}.
\end{equation*}

\begin{rem}
It is Matsuoka \cite{M} who defined $Z_k(s)$ and $f_k(s)$ as above, and he use the notation $f_k(s)$ as $Z_k(s)$ and does $h_k(s)$ as $f_k(s)$.
However, Y\i ld\i r\i m gave another definition of $Z_k(s)$, namely
\begin{equation*}
Z_k(s)=(\chi(s))^{\frac{1}{2}}\frac{d^k}{ds^k}((\chi(s))^{-\frac{1}{2}}\zeta(s)).
\end{equation*}
It is possible to see that this definition coincides with our definition.
The representation (\ref{Z_k-binom}) is inspired by the work of Matsumoto and Tanigawa \cite{M-T}.
\end{rem}

We give some properties of $Z_k(s)$.
\begin{lemma}\label{Z_k-fund}
For $k=0,1,2,\cdots$, $Z_k(s)$ has the following properties.
\begin{enumerate}
\item $Z^{(k)}(t)=i^k\chi \left(\frac{1}{2}+it \right)^{-\frac{1}{2}}Z_k\left(\frac{1}{2}+it \right)$.
\vspace{2.5mm}
\item $Z_k(s)$ satisfies the functional equation
\begin{equation}\label{fe-Z_k}
Z_k(s)=(-1)^k\chi(s)Z_k(1-s)
\end{equation}
for all $s$.
\end{enumerate}
\end{lemma}

\begin{proof}
When $k=0$, all the statements are trivial.
For $k=1$, see the proof of the lemma in \cite{C&G}. By induction, we can prove the both of two statements for $k\geq 1$.
In fact, when we derivate the equation in (i), we see that

\begin{align*}
&\quad
Z^{(k+1)}(t) \\
&=
i^{k+1}\left(\chi \left(\frac{1}{2}+it \right)^{-\frac{1}{2}}Z_k'\left(\frac{1}{2}+it \right) \right. \\
&\quad
\left. -\frac{1}{2}\chi'\left(\frac{1}{2}+it \right)\chi \left(\frac{1}{2}+it \right)^{-\frac{3}{2}}Z_k\left(\frac{1}{2}+it \right) \right) \\
&=
i^{k+1}\chi \left(\frac{1}{2}+it \right)^{-\frac{1}{2}}\left(Z_k'\left(\frac{1}{2}+it \right)-\frac{1}{2}\omega \left(\frac{1}{2}+it \right)Z_k\left(\frac{1}{2}+it \right) \right) \\
&=
i^{k+1}\chi \left(\frac{1}{2}+it \right)^{-\frac{1}{2}}Z_{k+1}\left(\frac{1}{2}+it \right).
\end{align*}

And by the definition, we have
\begin{align*}
&\quad
Z_{k+1}(s) \\
&=
Z_k'(s)-\frac{1}{2}\omega(s)Z_k(s) \\
&=(-1)^k\chi'(s)Z_k(1-s)-(-1)^k\chi(s)Z_k'(1-s) \\
&\quad
-\frac{1}{2}\omega(s)(-1)^k\chi(s)Z_k(1-s) \\
&=
(-1)^{k+1}\chi(s)\left(-\omega(s)Z_k(1-s)+Z_k'(1-s)+\frac{1}{2}\omega(s)Z_k(1-s)\right) \\
&=
(-1)^{k+1}\chi(s)Z_{k+1}(1-s).
\end{align*}
\end{proof}

This lemma is proved by Matsuoka \cite{M}, but we reproduce the proof here because \cite{M} is unpublished.

There are some results on the zeros of higher derivatives of Hardy's $Z$-function.
Let $N(T,Z_k)$ be the number of zeros of $Z_k(s)$ in the region
$\{ s=\sigma+it \mid 1-2m<\sigma<2m, 0\leq t\leq T \}$,
where $m=m(k)$ is a sufficiently large positive integer.
\begin{lemma}\label{M-T}
For any non-negative $k$ we have
\begin{equation*}
N(T,Z_k)=\frac{T}{2\pi}\log \frac{T}{2\pi}-\frac{T}{2\pi}+O_k(\log T).
\end{equation*}
\end{lemma}
This is Theorem 3 in \cite{M-T} essentially.
Actually, Matsumoto and Tanigawa proved this statement for $\eta_k(s)$, where
\begin{equation*}
\eta_1(s)=\zeta(s)-\frac{2}{\omega(s)}\zeta'(s)
\end{equation*}
and
\begin{equation*}
\eta_{k+1}(s)=\left(\frac{\omega'}{\omega}(s)-\frac{1}{2}\omega(s) \right)\eta_k(s)+\eta_k'(s).
\end{equation*}
We can see that $-\frac{1}{2}\omega(s)\eta_k(s)=Z_k(s)$ by induction. The case $k=1$ is trivial.
We assume the case $k\geq 1$ is true. Then
\begin{align*}
-\frac{1}{2}\omega(s)\eta_{k+1}(s)
&=
\left(-\frac{1}{2}\omega'(s)+\frac{1}{4}\omega^2(s) \right)\eta_k(s)-\frac{1}{2}\omega(s)\eta'_k(s) \\
&=
-\frac{1}{2}(\omega'(s)\eta_k(s)+\omega(s)\eta'_k(s))-\frac{1}{2}\omega(s)\left(-\frac{1}{2}\omega(s)\eta_k(s) \right) \\
&=
Z_k'(s)-\frac{1}{2}\omega(s)Z_k(s)=Z_{k+1}(s).
\end{align*}

\begin{lemma}[see the proof of Theorem 2 in \cite{M-T}]\label{M-T-1}
Assuming RH, $Z_k(s)$ has at most $O_k(1)$ zeros with ordinates in $(0,T)$ off the critical line.
\end{lemma}
They proved that the difference of the number of zeros of $Z^{(k)}(t)$ and that of $\eta_k(s)$ is $O_k(1)$.

\begin{lemma}[Lemma 4 in \cite{Y}]\label{distance}
Assuming RH, the zeros of $Z_k(s)$ which are not on $\sigma=\frac{1}{2}$ are within a distance $\frac{1}{9}$ from the line $\sigma=\frac{1}{2}$.
\end{lemma}

From Lemma \ref{M-T}, we see that there exists a sequence of positive numbers $\{ T_r \}_{r=1}^{\infty}(T_r\rightarrow \infty \ as \ r\rightarrow \infty)$ such that if $Z_k(\beta_k+i\gamma_k)=0$ then $|\gamma_k-T_r|^{-1}=O_k(\log T_r)$.
Moreover, Lemma \ref{M-T-1} says that for sufficiently $T_0=T_0(k)$, all zeros of $Z_k(s)$ for $t>T_0$ is on the critical line.
When we take $T$, we understand that it is $>T_0$ and in $\{ T_r \}_{r=1}^{\infty}$ hereafter.

$\mathscr{Z}_k(s,T)$ has important properties for our purpose.

\begin{lemma}[Lemma 5 in \cite{Y}]\label{app-log-der}
Assuming RH, we have
\begin{equation*}
\frac{Z_k'}{Z_k}(s)-\frac{\mathscr{Z}_k'}{\mathscr{Z}_k}(s,T)\ll \frac{U}{T}
\end{equation*}
for $\sigma \geq \frac{5}{8}$ and $T\leq t\leq T+U\leq 2T$.
\end{lemma}

\begin{lemma}[Lemma 6 in \cite{Y}]\label{poles-zeros}
We assume RH and let $k\geq 1$. At $s=1$ $\mathscr{Z}_k(s,T)$ has a pole of order $k+1$. There are $k$ zeros of $\mathscr{Z}_k(s,T)$ located at $z_g=1-\frac{2}{L}\theta_g+O_k\left(\frac{1}{L^2}\right) \ (g=1,\dots , k)$, where $\theta_g$'s are the roots of $\sum_{\mu=0}^{k}\frac{\theta^{\mu}}{\mu !}=0$. There are no other zeros or poles of $\mathscr{Z}_k(s,T)$ with $\frac{5}{8}\leq \sigma \leq 2$. Thus we have
\begin{equation*}
\frac{\mathscr{Z}_k'}{\mathscr{Z}_k}(s,T)=\frac{-(k+1)}{s-1}+\sum_{g=1}^{k}\frac{1}{s-z_g}+W(s,T)
\end{equation*}
where $W(s,T)$ is regular for $\frac{5}{8}\leq \sigma \leq \frac{9}{8}$.
\end{lemma}

\begin{lemma}\label{Dirichlet-app}
For $\sigma \geq \frac{9}{8}$, there is an absolutely convergent Dirichlet series such that
\begin{equation*}
\frac{\mathscr{Z}_k'}{\mathscr{Z}_k}(s,T)=\sum_{m=1}^{\infty}\frac{a_k(m)}{m^s}+O(T^{-1})
\end{equation*}
where, as $T\rightarrow \infty$, $a_k(m)=a_k(m,L)\ll_{\varepsilon} T^{\varepsilon}$ for any $\varepsilon>0$ and $m \ll T$.
\end{lemma}

\begin{proof}
This result has been proved in \cite{C&G2}
\end{proof}

Under the RH, we can obtain
\begin{equation*}
\frac{\zeta'}{\zeta}(s)\ll ((\log (|t|+2))^{2-2\sigma}+1)\min \left(\frac{1}{|\sigma-1|},\log \log (|t|+2) \right)
\end{equation*}
uniformly for $1/2+1/\log \log (|t|+2)\leq \sigma \leq 3/2, |t|\geq 1$ (see \cite{Mon&Vau}, p.435).
We can see that
\begin{equation*}
\frac{\zeta^{(\mu+1)}}{\zeta}(s)=\frac{d}{ds}\frac{\zeta^{(\mu)}}{\zeta}(s)+\frac{\zeta^{(\mu)}}{\zeta}(s)\frac{\zeta'}{\zeta}(s).
\end{equation*}
Hence, inductively applying Cauchy's integral theorem in a disk of radius $(\log (|t|+2))^{-1}$ around $s$, we have
\begin{equation*}
\frac{\zeta^{(\mu+1)}}{\zeta}(s)\ll_{\mu}((\log (|t|+2))^{\mu+2-2\sigma}+(\log (|t|+2))^{\mu})\log \log (|t|+2)
\end{equation*}
uniformly for $5/8\leq \sigma \leq 9/8, |t|\geq 2$. Therefore

\begin{align*}
\frac{\mathscr{Z}_k'}{\mathscr{Z}_k}(s,T)
&=
\frac{\sum_{\mu=0}^{k}\binom{k}{\mu}(\frac{L}{2})^{k-\mu}\zeta^{(\mu+1)}(s)}{\sum_{\mu=0}^{k}\binom{k}{\mu}(\frac{L}{2})^{k-\mu}\zeta^{(\mu)}(s)} \\
&=
\frac{\sum_{\mu=0}^{k}\binom{k}{\mu}(\frac{L}{2})^{-\mu}\frac{\zeta^{(\mu+1)}}{\zeta}(s)}{1+\sum_{\mu=1}^{k}\binom{k}{\mu}(\frac{L}{2})^{-\mu}\frac{\zeta^{(\mu)}}{\zeta}(s)} \\
&=
\sum_{\mu=0}^{k}\binom{k}{\mu}\left(\frac{L}{2}\right)^{-\mu}\frac{\zeta^{(\mu+1)}}{\zeta}(s)(1+o(1)) \\
&\ll_{k,\varepsilon}
|t|^{\varepsilon}
\end{align*}
uniformly for $5/8\leq \sigma \leq 9/8, |t|\geq 2$.

As in the paper of Y\i ld\i r\i m \cite{Y}, we apply the following lemma by Gonek \cite{Go} :
\begin{lemma}[Gonek]\label{G'slem}
Let $a>1$ be fixed and let $m$ be a non-negative integer.
Let the Dirichlet series $\sum_{n=1}^{\infty}b_n n^{-a-it}$ be absolutely convergent with a sequence of complex number $\{b_n\}_{n=1}^{\infty}$.
Then for any sufficiently large $T$,
\begin{equation*}
\begin{split}
&\quad \frac{1}{2\pi}\int_{1}^{T}\left(\sum_{n=1}^{\infty}b_n n^{-a-it} \right)\chi(1-a-it)
\left(\log \frac{t}{2\pi} \right)^m dt \\
&= \sum_{1\leq n\leq T/2\pi} b_n(\log n)^m+O(T^{a-\frac{1}{2}}(\log T)^m).
\end{split}
\end{equation*}
\end{lemma}

Finally, we introduce some fundamental lemmas. Stirling's formula implies
\begin{lemma}
For $-1<\sigma<2$ and $t\geq1$, we have
\begin{equation}
\chi(1-s)=e^{-\frac{\pi i}{4}}\left(\frac{t}{2\pi} \right)^{\sigma-\frac{1}{2}}\exp \left(it\log \frac{t}{2\pi e}\right) \left(1+O\left(\frac{1}{t}\right)\right),
\end{equation}
\begin{equation}\label{log-stirling}
\frac{\chi'}{\chi}(s)=-\log \frac{t}{2\pi}+O\left(\frac{1}{t} \right),
\end{equation}
and
\begin{equation*}
\left(\frac{\chi'}{\chi}\right)^{(k)}(s)=O_k\left(\frac{1}{t} \right).
\end{equation*}
\end{lemma}

Then, by the definition, for $-1<\sigma<2$ and $t\geq1$

\begin{equation}\label{f_k-ap}
f_k(s)=\left(\frac{1}{2}\log \frac{t}{2\pi} \right)^k+O_k\left(t^{-1}(\log t)^{k-1} \right).
\end{equation}

If RH is true, then the Lindel\"{o}f Hypohesis is also true.
Therefore we can obtain the following estimates. 
\begin{lemma}\label{lindelof}
If the RH is true, then for $\mu=0,1,2,\dots$ and $|t|\geq1$,
\begin{equation*}
\zeta^{(\mu)}(s) \ll_{\mu,\varepsilon}
\begin{cases}
1 & \text{$1< \sigma$,} \\
|t|^{\varepsilon} & \text{$\frac{1}{2}\leq \sigma \leq1$,} \\
|t|^{\frac{1}{2}-\sigma+\varepsilon} & \text{$-1<\sigma <\frac{1}{2}$.}
\end{cases}
\end{equation*}
\end{lemma}
When $\mu=0$, these estimates are well-known. For $\mu \geq 1$, we can obtain this estimates, using Cauchy's theorem in a disk of radius $(\log (|t|+2))^{-1}$ around $s$.

By (\ref{Z_k-binom}), this lemma leads to
\begin{equation*}
Z_k(s)\ll_{k,\varepsilon}
\begin{cases}
|t|^{\varepsilon} & \text{$\frac{1}{2}\leq \sigma<2$,} \\
|t|^{\frac{1}{2}-\sigma+\varepsilon} & \text{$-1<\sigma <\frac{1}{2}$}
\end{cases}
\end{equation*}
for $|t|\geq 1$.

Now we can show that
\begin{equation*}
\frac{Z_k'}{Z_k}(\sigma+iT)=O_k((\log T)^2)
\end{equation*}
uniformly for $-1\leq \sigma \leq 2$
by applying the following lemma
\begin{lemma}[Lemma $\alpha$ in \cite{T}]
If $f(s)$ is regular, and
\begin{equation*}
\left|\frac{f(s)}{f(s_0)} \right|<e^M \quad (M>1)
\end{equation*}
in the circle $|s-s_0|\leq r$, then
\begin{equation*}
\left|\frac{f'(s)}{f(s)}-\sum_{\rho}\frac{1}{s-\rho}\right|<\frac{AM}{r}
\quad \left(|s-s_0|\leq \frac{r}{4}\right),
\end{equation*}
where $\rho$ runs over the zeros of $f(s)$ such that $|\rho-s_0|\leq r/2$ and $A$ is an absolute positive constant.
\end{lemma}
We use this lemma with $f(s)=Z_k(s), r=12$ and $s_0=2+iT$.
The estimate of $Z_k(s)$ implies that we can take $M=\log T$ in this lemma.
Hence we have
\begin{equation*}
\frac{Z_k'}{Z_k}(\sigma+iT)=\sum_{\substack{\rho \\ |\rho-(2+iT)|\leq 6}}\frac{1}{s-\rho}+O_k(\log T).
\end{equation*}
By the way of taking $T$ and Lemma \ref{M-T}, we see that
\begin{align*}
\sum_{\substack{\rho \\ |\rho-(2+iT)|\leq 6}}\frac{1}{s-\rho}
&\ll_k
\sum_{\substack{\rho \\ |\rho-(2+iT)|\leq 6}}\log T \\
&\ll_k
(\log T)^2.
\end{align*}

\section{The proof of the theorem}
Our proof is inspired by the proof of Y\i ld\i r\i m \cite{Y}.
As we mentioned before, we consider sufficiently large $T$ in $\{ T_r \}_{r=1}^{\infty}$.
This restriction will be removed at the end of the proof.
Now by Lemma \ref{M-T-1}, $Z_k(s)$ has at most $O_k(1)$ zeros off the critical line up to $T$.
At such a zero, by Lemma \ref{distance},
\begin{equation*}
|Z_j(\rho_k)|^2\ll_{j,\varepsilon} |\Im \rho_k|^{\frac{2}{9}+\varepsilon},
\end{equation*}
whence
\begin{equation*}
\sum_{\substack{0<\Im \rho_k \leq T\\ \Re \rho_k \neq \frac{1}{2}}}|Z_j(\rho_k)|^2\ll_{j,k,\varepsilon} T^{\frac{2}{9}+\varepsilon},
\end{equation*}
where $\rho_k$ is the zeros of $Z_k(s)$.
Therefore, by Lemma \ref{Z_k-fund},
\begin{align*}
\sum_{0<\gamma_k\leq T}\left|Z^{(j)}(\gamma_k) \right|^2
&=\sum_{0<\gamma_k\leq T}\left|Z_j\left(\frac{1}{2}+i\gamma_k\right) \right|^2 \\
&=\sum_{\substack{\rho_k\\0<\Im \rho_k \leq T}}\left|Z_j(\rho_k) \right|^2+O_{j,k,\varepsilon}(T^{\frac{2}{9}+\varepsilon}) \\
&=M(T)+O_{j,k,\varepsilon}(T^{\frac{2}{9}+\varepsilon}),
\end{align*}
say.
For the convenience, we consider a sum over shorter range.
Let
\begin{equation*}
U=T^{\frac{3}{4}}
\end{equation*}
and let $R$ be the positively oriented rectangular path with vertices $c+iT$, $c+i(T+U)$, $1-c+i(T+U)$ and $1-c+iT$, where $c=\frac{5}{8}$.
Then we need to consider
\begin{equation*}
M(T+U)-M(T)=\frac{1}{2\pi i}\int_{R}\frac{Z_k'}{Z_k}(s)Z_j(s)Z_j(1-s)ds.
\end{equation*}

On the horizontal line, since
\begin{equation*}
\frac{Z_k'}{Z_k}(s)\ll_{k,\varepsilon} T^{\varepsilon} \quad \text{and} \quad Z_k(s)Z_k(1-s)\ll_{k,\varepsilon}T^{c-\frac{1}{2}+\varepsilon},
\end{equation*}
we can see that
\begin{equation*}
\int_{1-c+iT}^{c+iT}\frac{Z_k'}{Z_k}(s)Z_j(s)Z_j(1-s)ds\ll_{k,\varepsilon}T^{c-\frac{1}{2}+\varepsilon}.
\end{equation*}
Thus
\begin{align*}
&\quad
\frac{1}{2\pi i}\int_{R}\frac{Z_k'}{Z_k}(s)Z_j(s)Z_j(1-s)ds \\
&=\frac{1}{2\pi i}\int_{c+iT}^{c+i(T+U)}\frac{Z_k'}{Z_k}(s)Z_j(s)Z_j(1-s)ds \\
&\quad +\frac{1}{2\pi i}\int_{1-c+i(T+U)}^{1-c+iT}\frac{Z_k'}{Z_k}(s)Z_j(s)Z_j(1-s)ds+O_{j,k,\varepsilon}\left(T^{c-\frac{1}{2}+\varepsilon} \right) \\
&= I_{1}+I_{2}+O_{j,k,\varepsilon}\left(T^{c-\frac{1}{2}+\varepsilon} \right),
\end{align*}
say.
On the integral $I_2$,
\begin{align*}
I_2&= -\frac{1}{2\pi i}\int_{1-c+iT}^{1-c+i(T+U)}\frac{Z_k'}{Z_k}(s)Z_j(s)Z_j(1-s)ds \\
    &= -\frac{1}{2\pi i}\int_{1-c+iT}^{1-c+i(T+U)}\left(\frac{\chi'}{\chi}(s)-\frac{Z_k'}{Z_k}(1-s)\right)Z_j(s)Z_j(1-s)ds \\
    &= -\frac{1}{2\pi i}\int_{1-c+iT}^{1-c+i(T+U)}\frac{\chi'}{\chi}(s)Z_j(s)Z_j(1-s)ds \\
    &\quad +\frac{1}{2\pi i}\int_{1-c+iT}^{1-c+i(T+U)}\frac{Z_k'}{Z_k}(1-s)Z_j(s)Z_j(1-s)ds.
\end{align*}
When we replace $s$ by $1-s$, the second integral is 
\begin{equation*}
-\frac{1}{2\pi i}\int_{c-iT}^{c-i(T+U)}\frac{Z_k'}{Z_k}(s)Z_j(s)Z_j(1-s)ds= \overline{I_1}.
\end{equation*}
Now we see that
\begin{align*}
&\quad M(T+U)-M(T) \\
&=-\frac{1}{2\pi i}\int_{1-c+iT}^{1-c+i(T+U)}\frac{\chi'}{\chi}(s)Z_j(s)Z_j(1-s)ds+2\Re{I_1}+O_{j,k,\varepsilon}\left(T^{c-\frac{1}{2}+\varepsilon} \right).
\end{align*}

We divide the following argument into 5 steps;
\begin{description}
\item[Step 1]\mbox{}\\
Calculate the integral
\begin{equation*}
-\frac{1}{2\pi i}\int_{1-c+iT}^{1-c+i(T+U)}\frac{\chi'}{\chi}(s)Z_j(s)Z_j(1-s)ds,
\end{equation*}

\item[Step 2]\mbox{}\\
Transform the integral $I_1$ to certain sums of arithmetic functions,

\item[Step 3]\mbox{}\\
To derive some approximate formula for those sums by Perron's formula,

\item[Step 4]\mbox{}\\
Express $I_1$ with that formula and simplify the coefficients,

\item[Step 5]\mbox{}\\
Concluding.
\end{description}

\subsection*{Step 1}
By Cauchy's integral theorem, the integral is equal to
\begin{equation*}
-\frac{1}{2\pi i}\int_{\frac{1}{2}+iT}^{\frac{1}{2}+i(T+U)}\frac{\chi'}{\chi}(s)Z_j(s)Z_j(1-s)ds+O_{j,\varepsilon}\left(T^{c-\frac{1}{2}+\varepsilon} \right).
\end{equation*}
From (\ref{log-stirling}) and Lemma \ref{lindelof} we see that the above integral is
\begin{equation}\label{mean-value}
\frac{1}{2\pi} \int_{T}^{T+U}\log \frac{t}{2\pi} Z^{(j)}(t)^2dt +O_{j,\varepsilon}\left(T^{\varepsilon} \right).
\end{equation}
Here we put
\begin{equation*}
Y_j(t)=\int_{1}^{t}Z^{(j)}(x)^2dx.
\end{equation*}
Using integration by parts and the result of Minamide and Tanigawa, we can show that the integral in (\ref{mean-value}) is equal to
\begin{equation*}
\begin{split}
&\quad \frac{1}{2\pi} \log \frac{T+U}{2\pi}Y_j(T+U)-\frac{1}{2\pi} \log \frac{T}{2\pi}Y_j(T)-\frac{1}{2\pi}\int_{T}^{T+U}t^{-1}Y_j(t)dt \\
&=\frac{T+U}{2\cdot 4^j(2j+1)\pi}P_{2j+1}\left(\log \frac{T+U}{2\pi}\right)\log \frac{T+U}{2\pi} \\
&\quad -\frac{T}{2\cdot 4^j(2j+1)\pi}P_{2j+1}\left(\log \frac{T}{2\pi}\right)\log \frac{T}{2\pi} \\
&\quad -\frac{1}{2\cdot 4^j(2j+1)\pi}\int_{T}^{T+U}\left \{ P_{2j+1}\left(\log \frac{t}{2\pi}\right)+O(t^{-\frac{1}{2}}\log^{2j+1}t)\right \}dt \\
&\quad +O_j(T^{\frac{1}{2}}\log^{2j+1}T) \\
&=\frac{U}{2\cdot 4^j(2j+1)\pi}\left(\log \frac{T}{2\pi}\right)^{2j+2}+O_{j}\left(U(\log T)^{2j+1} \right),
\end{split}
\end{equation*}
because
\begin{equation*}
\log \frac{T+U}{2\pi}=\log \frac{T}{2\pi}\left(1+O\left(\frac{U}{T\log T}\right)\right).
\end{equation*}

\subsection*{Step 2}
We calculate $I_1$. By the functional equation (\ref{fe-Z_k}) and Lemma \ref{app-log-der}, we have
\begin{align*}
I_1
&=\frac{(-1)^j}{2\pi i}\int_{c+iT}^{c+i(T+U)}\frac{Z_k'}{Z_k}(s)Z_j(s)^2\chi(1-s)ds \\
&=\frac{(-1)^j}{2\pi i}\int_{c+iT}^{c+i(T+U)}\frac{\mathscr{Z}_k'}{\mathscr{Z}_k}(s,T)Z_j(s)^2\chi(1-s)ds+O_{j,k,\varepsilon}\left(U^2T^{c-\frac{3}{2}+\varepsilon} \right).
\end{align*}
The representation of (\ref{Z_k-binom}) and the approximation of $f_k(s)$ (\ref{f_k-ap}) imply that the above is
\begin{align*}
&=\frac{(-1)^j}{2\pi i}\int_{c+iT}^{c+i(T+U)}\frac{\mathscr{Z}_k'}{\mathscr{Z}_k}(s,T)\left(\sum_{\mu=0}^{j}\binom{j}{\mu}f_{j-\mu}(s)\zeta^{(\mu)}(s) \right)^2\chi(1-s)ds \\
&\quad+O_{j,k,\varepsilon}\left(U^2T^{c-\frac{3}{2}+\varepsilon} \right) \\
&=\frac{(-1)^j}{2\pi i}\int_{c+iT}^{c+i(T+U)}\frac{\mathscr{Z}_k'}{\mathscr{Z}_k}(s,T)\left(\sum_{\mu=0}^{j}\binom{j}{\mu}\left(\frac{1}{2}\log \frac{t}{2\pi} \right)^{j-\mu}\zeta^{(\mu)}(s) \right)^2\chi(1-s)ds \\
&\quad +O_{j,k,\varepsilon}\left(U^2T^{c-\frac{3}{2}+\varepsilon} \right)+O_{j,k,\varepsilon}\left(T^{c-\frac{1}{2}+\varepsilon} \right) \\
&=\frac{(-1)^j}{2\pi i}\int_{b+iT}^{b+i(T+U)}\frac{\mathscr{Z}_k'}{\mathscr{Z}_k}(s,T)\left(\sum_{\mu=0}^{j}\binom{j}{\mu}\left(\frac{1}{2}\log \frac{t}{2\pi} \right)^{j-\mu}\zeta^{(\mu)}(s) \right)^2\chi(1-s)ds \\
&\quad +O_{j,k,\varepsilon}\left(U^2T^{c-\frac{3}{2}+\varepsilon} \right)+O_{j,k,\varepsilon}\left(T^{b-\frac{1}{2}+\varepsilon} \right),
\end{align*}
where $b=\frac{9}{8}$.
To show the last equality, we use Cauchy's integral theorem.
We note that
\begin{equation}\label{square}
\begin{split}
&\quad \left(\sum_{\mu=0}^{j}\binom{j}{\mu}\left(\frac{1}{2}\log \frac{t}{2\pi} \right)^{j-\mu}\zeta^{(\mu)}(s) \right)^2 \\
&=\sum_{\mu=0}^{j}\sum_{\nu=0}^{j}\binom{j}{\mu}\binom{j}{\nu}\left(\frac{1}{2}\log \frac{t}{2\pi} \right)^{2j-\mu-\nu}\zeta^{(\mu)}(s)\zeta^{(\nu)}(s).
\end{split}
\end{equation}

Therefore, by Lemma \ref{G'slem}, our problem is reduced to consider
\begin{equation}\label{arithmetic_sum}
\begin{split}
&\quad
\frac{1}{2\pi i}\int_{b+iT}^{b+i(T+U)}\frac{\mathscr{Z}_k'}{\mathscr{Z}_k}(s,T)\zeta^{(\mu)}(s)\zeta^{(\nu)}(s)\chi(1-s)\left(\log \frac{t}{2\pi} \right)^{2j-\mu-\nu}ds \\
&=
\frac{1}{2\pi i}\int_{b+iT}^{b+i(T+U)}\sum_{m=1}^{\infty}\frac{a_k(m)}{m^s}\zeta^{(\mu)}(s)\zeta^{(\nu)}(s)\chi(1-s)\left(\log \frac{t}{2\pi} \right)^{2j-\mu-\nu}ds \\
&\quad
+O_{\mu,\nu,k,\varepsilon}\left(T^{b-\frac{1}{2}+\varepsilon} \right) \\
&=
\sum_{\frac{T}{2\pi}\leq mn\leq \frac{T+U}{2\pi}}a_k(m)D_{\mu \nu}(n)(\log mn)^{2j-\mu-\nu}+O_{\mu,\nu,k,\varepsilon}\left(T^{b-\frac{1}{2}+\varepsilon} \right),
\end{split}
\end{equation}
where $D_{\mu \nu}(n)$ satisfies

\begin{equation*}
\zeta^{(\mu)}(s)\zeta^{(\nu)}(s)=\sum_{n=1}^{\infty}\frac{D_{\mu \nu}(n)}{n^s}
\end{equation*}
for $\sigma>1$.
If we can calculate the sum
\begin{equation*}
\sum_{mn\leq x}a_k(m)D_{\mu \nu}(n),
\end{equation*}
then by partial summation we are able to compute the sum on the right-hand side in (\ref{arithmetic_sum}).

\subsection*{Step 3}
By Perron's formula,
\begin{equation*}
\sum_{mn\leq x}a_k(m)D_{\mu \nu}(n)=\frac{1}{2\pi i}\int_{b-iT}^{b+iT}\sum_{m=1}^{\infty}\frac{a_k(m)}{m^s}\zeta^{(\mu)}(s)\zeta^{(\nu)}(s)\frac{x^s}{s}ds+O(x^{\varepsilon})+R,
\end{equation*}
where $R$ is the error term appearing in Perron's formula (see \cite[p.140]{Mon&Vau}) which satisfies that
\begin{align*}
R
&\ll
\sum_{\substack{x/2<mn<2x \\ mn \neq x}} |a_k(m)D_{\mu \nu}(n)| \min \left(1, \frac{x}{T|x-mn|}\right) \\
&\quad
+\frac{(4x)^{b}}{T}\sum_{mn=1}^{\infty}\frac{|a_k(m)D_{\mu \nu}(n)|}{(mn)^{b}}.
\end{align*}
On the first term, we see that
\begin{align*}
&\quad
\sum_{\substack{x/2<mn<2x \\ n \neq x}} |a_k(m)D_{\mu \nu}(n)| \min \left(1, \frac{x}{T|x-mn|}\right) \\
&\ll
\frac{x}{T}\sum_{x/2<mn<x-1}\left|\frac{a_k(m)D_{\mu \nu}(n)}{x-mn} \right|+\sum_{x-1\leq mn\leq x+1} |a_k(m)D_{\mu \nu}(n)| \\
&\quad
+\frac{x}{T}\sum_{x+1<mn<2x}\left|\frac{a_k(m)D_{\mu \nu}(n)}{x-mn} \right| \\
&=
\frac{x}{T}\sum_{x/2<l<x-1}\sum_{l=mn}\left|\frac{a_k(m)D_{\mu \nu}(n)}{x-l} \right|+\sum_{x-1\leq l\leq x+1}\sum_{l=mn} |a_k(m)D_{\mu \nu}(n)| \\
&\quad
+\frac{x}{T}\sum_{x+1<l<2x}\sum_{l=mn}\left|\frac{a_k(m)D_{\mu \nu}(n)}{x-l} \right| \\
&\ll_{\mu,\nu}
\frac{x^{1+\varepsilon}}{T}\sum_{x/2<l<x-1}\sum_{l=mn}\frac{1}{x-l}+x^{\varepsilon}\sum_{x-1\leq l\leq x+1}\sum_{l=mn} 1 \\
&\quad
+\frac{x^{1+\varepsilon}}{T}\sum_{x+1<l<2x}\sum_{l=mn}\frac{1}{l-x} \\
&=
\frac{x^{1+\varepsilon}}{T}\sum_{x/2<l<x-1}\frac{d(l)}{x-l}+x^{\varepsilon}\sum_{x-1\leq l\leq x+1}d(l)+\frac{x^{1+\varepsilon}}{T}\sum_{x+1<l<2x}\frac{d(l)}{l-x} \\
&\ll_{\varepsilon}
\frac{x^{1+\varepsilon}}{T}\sum_{x/2<l<x-1}\frac{1}{x-l}+x^{\varepsilon}+\frac{x^{1+\varepsilon}}{T}\sum_{x+1<l<2x}\frac{1}{l-x} \\
&\ll
\frac{x^{1+\varepsilon}}{T}\sum_{1<l<x}\frac{1}{l}+x^{\varepsilon} \ll \frac{x^{1+\varepsilon}}{T}+x^{\varepsilon}.
\end{align*}
Therefore we obtain
\begin{equation*}
R\ll_{\mu,\nu,\varepsilon} \frac{x^{b}}{T}+x^{\varepsilon}.
\end{equation*}

By using Lemmas \ref{poles-zeros}, \ref{Dirichlet-app} and the residue theorem, we have
\begin{align*}
&\quad
\frac{1}{2\pi i}\int_{b-iT}^{b+iT}\sum_{m=1}^{\infty}\frac{a_k(m)}{m^s}\zeta^{(\mu)}(s)\zeta^{(\nu)}(s)\frac{x^s}{s}ds \\
&=
\frac{1}{2\pi i}\int_{b-iT}^{b+iT}\frac{\mathscr{Z}_k'}{\mathscr{Z}_k}(s,T)\zeta^{(\mu)}(s)\zeta^{(\nu)}(s)\frac{x^s}{s}ds+O_{\mu,\nu,\varepsilon}(x^{b}T^{-1+\varepsilon}) \\
&=
\underset{s=1}{\mathrm {Res}}\ \frac{\mathscr{Z}_k'}{\mathscr{Z}_k}(s,T)\zeta^{(\mu)}(s)\zeta^{(\nu)}(s)\frac{x^s}{s}+\sum_{g=1}^{k}\underset{s=z_g}{\mathrm {Res}}\ \frac{\mathscr{Z}_k'}{\mathscr{Z}_k}(s,T)\zeta^{(\mu)}(s)\zeta^{(\nu)}(s)\frac{x^s}{s} \\
&\quad
+\frac{1}{2\pi i}\int_{c-iT}^{c+iT}\frac{\mathscr{Z}_k'}{\mathscr{Z}_k}(s,T)\zeta^{(\mu)}(s)\zeta^{(\nu)}(s)\frac{x^s}{s}ds+O_{\mu,\nu,\varepsilon}(x^{b}T^{-1+\varepsilon}) \\
&=
\underset{s=1}{\mathrm {Res}}\ \frac{\mathscr{Z}_k'}{\mathscr{Z}_k}(s,T)\zeta^{(\mu)}(s)\zeta^{(\nu)}(s)\frac{x^s}{s}+\sum_{g=1}^{k}\underset{s=z_g}{\mathrm {Res}}\ \frac{\mathscr{Z}_k'}{\mathscr{Z}_k}(s,T)\zeta^{(\mu)}(s)\zeta^{(\nu)}(s)\frac{x^s}{s} \\
&\quad
+O_{\mu,\nu,\varepsilon}(x^{c}T^{\varepsilon}+x^{b}T^{-1+\varepsilon}).
\end{align*}

To calculate the residues, we note that
\begin{equation*}
\frac{\mathscr{Z}_k'}{\mathscr{Z}_k}(s,T)=\frac{-(k+1)}{s-1}+\sum_{g=1}^{k}\frac{1}{s-z_g}+W(s,T),
\end{equation*}
\begin{equation*}
\frac{x^s}{s}=x\sum_{l=0}^{\infty}\left(\sum_{r=0}^{l}\frac{(-1)^r}{(l-r)!}(\log y)^{l-r} \right)(s-1)^l
\end{equation*}
and
\begin{equation*}
\zeta^{(\mu)}(s)=\frac{(-1)^{\mu}\mu !}{(s-1)^{\mu+1}}+\sum_{n=\mu}^{\infty}\frac{n!}{(n-\mu)!}c_n(s-1)^{n-\mu},
\end{equation*}
where $c_n$ is the $n$-th Stieltjes constant as in (\ref{c_h}).

On the residue at $s=z_g$, we have
\begin{align*}
&\quad
\sum_{g=1}^{k}\underset{s=z_g}{\mathrm {Res}}\ \frac{\mathscr{Z}_k'}{\mathscr{Z}_k}(s,T)\zeta^{(\mu)}(s)\zeta^{(\nu)}(s)\frac{x^s}{s} \\
&=
\sum_{g=1}^{k}\zeta^{(\mu)}(z_g)\zeta^{(\nu)}(z_g)\frac{x^{z_g}}{z_g} \\
&=
\sum_{g=1}^{k}\frac{x^{z_g}}{z_g} \left \{ \frac{(-1)^{\mu+\nu}\mu !\nu !}{(z_g-1)^{\mu+\nu+2}}+(-1)^{\mu} \mu !\sum_{n=\mu}^{\infty}\frac{n!}{(n-\mu)!}c_n(z_g-1)^{n-\mu-\nu-1} \right. \\
&
\left. \hspace{2cm}+(-1)^{\mu}\mu !\sum_{m=\nu}^{\infty}\frac{m!}{(m-\nu)!}c_m(z_g-1)^{m-\mu-\nu-1} \right. \\
&
\left. \hspace{2cm}+\sum_{n=\mu}^{\infty}\sum_{m=\nu}^{\infty}\frac{n!m!c_nc_m}{(n-\mu)!(m-\nu)!}(z_g-1)^{m+n-\mu-\nu} \right \},
\end{align*}
because
\begin{align*}
\zeta^{(\mu)}(s)\zeta^{(\nu)}(s)
&=
\frac{(-1)^{\mu+\nu}\mu !\nu !}{(s-1)^{\mu+\nu+2}}+(-1)^{\nu}\nu !\sum_{n=\mu}^{\infty}\frac{n!}{(n-\mu)!}c_n(s-1)^{n-\mu-\nu-1} \\
&\quad
+(-1)^{\mu}\mu !\sum_{m=\nu}^{\infty}\frac{m!}{(m-\nu)!}c_m(s-1)^{m-\mu-\nu-1} \\
&\quad
+\sum_{n=\mu}^{\infty}\sum_{m=\nu}^{\infty}\frac{n!m!c_nc_m}{(n-\mu)!(m-\nu)!}(s-1)^{m+n-\mu-\nu}.
\end{align*}

Next we consider the residue at $s=1$. We see that
\begin{align*}
&\quad
\underset{s=1}{\mathrm {Res}}\ \frac{\mathscr{Z}_k'}{\mathscr{Z}_k}(s,T)\zeta^{(\mu)}(s)\zeta^{(\nu)}(s)\frac{x^s}{s} \\
&=
-\underset{s=1}{\mathrm {Res}}\frac{k+1}{s-1}\zeta^{(\mu)}(s)\zeta^{(\nu)}(s)\frac{x^s}{s}+\underset{s=1}{\mathrm {Res}}\ \sum_{g=1}^{k}\frac{1}{s-z_g}\zeta^{(\mu)}(s)\zeta^{(\nu)}(s)\frac{x^s}{s} \\
&\quad
+\underset{s=1}{\mathrm {Res}}\ W(s,T)\zeta^{(\mu)}(s)\zeta^{(\nu)}(s)\frac{x^s}{s}=R_1+R_2+R_3,
\end{align*}
say. Since
\begin{align*}
&\quad
\zeta^{(\mu)}(s)\zeta^{(\nu)}(s)\frac{x^s}{s} \\
&=
(-1)^{\mu+\nu}\mu !\nu !x\sum_{l=0}^{\infty}\left(\sum_{r=0}^{l}\frac{(-1)^r}{(l-r)!}(\log x)^{l-r} \right)(s-1)^{l-\mu-\nu-2} \\
&\quad
+(-1)^{\nu}\nu !x\sum_{n=\mu}^{\infty}\sum_{l=0}^{\infty}\frac{n!}{(n-\mu)!}c_n\left(\sum_{r=0}^{l}\frac{(-1)^r}{(l-r)!}(\log x)^{l-r}\right)(s-1)^{l+n-\mu-\nu-1} \\
&\quad
+(-1)^{\mu}\mu !x\sum_{m=\nu}^{\infty}\sum_{l=0}^{\infty}\frac{m!}{(m-\nu)!}c_m\left(\sum_{r=0}^{l}\frac{(-1)^r}{(l-r)!}(\log x)^{l-r}\right)(s-1)^{l+m-\mu-\nu-1} \\
&\quad
+x\sum_{n=\mu}^{\infty}\sum_{m=\nu}^{\infty}\sum_{l=0}^{\infty}\frac{n!m!c_n c_m}{(n-\mu)!(m-\nu)!}\left(\sum_{r=0}^{l}\frac{(-1)^r}{(l-r)!}(\log x)^{l-r}\right)(s-1)^{l+m+n-\mu-\nu},
\end{align*}
we have
\begin{align*}
R_1
&=
(-1)^{\mu+\nu+1}(k+1)\mu !\nu !x\sum_{r=0}^{\mu+\nu+2}\frac{(-1)^r}{(\mu+\nu+2-r)!}(\log x)^{\mu+\nu+2-r} \\
&\quad
+(-1)^{\nu+1}(k+1)\nu !x\sum_{n=\mu}^{\mu+\nu+1}\sum_{l=0}^{\mu+\nu+1-n}\frac{n!}{(n-\mu)!}c_n\sum_{r=0}^{l}\frac{(-1)^r}{(l-r)!}(\log x)^{l-r} \\
&\quad
+(-1)^{\mu+1}(k+1)\mu !x\sum_{m=\nu}^{\mu+\nu+1}\sum_{l=0}^{\mu+\nu+1-m}\frac{m!}{(m-\nu)!}c_m\sum_{r=0}^{l}\frac{(-1)^r}{(l-r)!}(\log x)^{l-r} \\
&\quad
-(k+1)\mu !\nu !c_{\mu}c_{\nu}x.
\end{align*}
We emphasise that the largest term is
\begin{equation*}
(k+1)\frac{(-1)^{\mu+\nu+1}\mu !\nu !}{(\mu+\nu+2)!}x(\log x)^{\mu+\nu+2}.
\end{equation*}
As for $R_2$,
\begin{align*}
\frac{1}{s-z_g}\zeta^{(\mu)}(s)\zeta^{(\nu)}(s)\frac{x^s}{s}
&=
(-1)^{\mu+\nu}\mu !\nu !x\sum_{\lambda=0}^{\infty}\sum_{l=0}^{\infty}\frac{(-1)^{\lambda}}{(1-z_g)^{\lambda+1}} \\
&\quad
\times \left(\sum_{r=0}^{l}\frac{(-1)^r}{(l-r)!}(\log x)^{l-r} \right)(s-1)^{\lambda+l-\mu-\nu-2} \\
&\quad
+(-1)^{\mu}\mu !x\sum_{\lambda=0}^{\infty}\sum_{n=\mu}^{\infty}\sum_{l=0}^{\infty}\frac{(-1)^{\lambda}n!}{(n-\mu)!(1-z_g)^{\lambda+1}}c_n \\
&\quad
\times \left(\sum_{r=0}^{l}\frac{(-1)^r}{(l-r)!}(\log x)^{l-r}\right)(s-1)^{\lambda+l+n-\mu-\nu-1} \\
&\quad
+(-1)^{\nu}\nu !x\sum_{\lambda=0}^{\infty}\sum_{m=\nu}^{\infty}\sum_{l=0}^{\infty}\frac{(-1)^{\lambda}m!}{(m-\nu)!(1-z_g)^{\lambda+1}}c_m \\
&\quad
\times \left(\sum_{r=0}^{l}\frac{(-1)^r}{(l-r)!}(\log x)^{l-r}\right)(s-1)^{\lambda+l+m-\mu-\nu-1} \\
&\quad
+x\sum_{\lambda=0}^{\infty}\sum_{n=\mu}^{\infty}\sum_{m=\nu}^{\infty}\sum_{l=0}^{\infty}\frac{(-1)^{\lambda}n!m!c_n c_m}{(n-\mu)!(m-\nu)!(1-z_g)^{\lambda+1}} \\
&\quad
\times \left(\sum_{r=0}^{l}\frac{(-1)^r}{(l-r)!}(\log x)^{l-r}\right)(s-1)^{\lambda+l+m+n-\mu-\nu},
\end{align*}
because
\begin{equation*}
\frac{1}{s-z_g}=\sum_{\lambda=0}^{\infty}\frac{(-1)^{\lambda}}{(1-z_g)^{\lambda+1}}(s-1)^{\lambda}.
\end{equation*}
Thus we have
\begin{align*}
&\quad
R_2 \\
&=
(-1)^{\mu+\nu}\mu !\nu !x\sum_{g=1}^{k}\sum_{\substack{\lambda+l=\mu+\nu+1 \\ 0\leq \lambda,l}}\frac{(-1)^{\lambda}}{(1-z_g)^{\lambda+1}}\left(\sum_{r=0}^{l}\frac{(-1)^r}{(l-r)!}(\log x)^{l-r} \right) \\
&\quad
+(-1)^{\mu}\mu !x\sum_{g=1}^{k}\sum_{\substack{\lambda+l+n=\mu+\nu \\ 0\leq \lambda,l \\ \mu \leq n}}\frac{(-1)^{\lambda}n!c_n}{(n-\mu)!(1-z_g)^{\lambda+1}}\left(\sum_{r=0}^{l}\frac{(-1)^r}{(l-r)!}(\log x)^{l-r}\right) \\
&\quad
+(-1)^{\nu}\nu !x\sum_{g=1}^{k}\sum_{\substack{\lambda+l+m=\mu+\nu \\ 0\leq \lambda,l \\ \nu \leq m}}\frac{(-1)^{\lambda}m!c_m}{(m-\nu)!(1-z_g)^{\lambda+1}}\left(\sum_{r=0}^{l}\frac{(-1)^r}{(l-r)!}(\log x)^{l-r}\right).
\end{align*}
The main term in our final result will come from the first sum with $r=0$, namely,
\begin{align*}
&\quad
(-1)^{\mu+\nu}\mu !\nu !x\sum_{g=1}^{k}\sum_{\substack{\lambda+l=\mu+\nu+1 \\ 0\leq \lambda,l}}\frac{(-1)^{\lambda}}{(1-z_g)^{\lambda+1}}\frac{(\log x)^{l}}{l!} \\
&=
(-1)^{\mu+\nu+1}\mu !\nu !x\sum_{g=1}^{k}\sum_{\lambda=0}^{\mu+\nu+1}\frac{1}{(\mu+\nu+1-\lambda)!}\frac{(\log x)^{\mu+\nu+1-\lambda}}{(z_g-1)^{\lambda+1}}.
\end{align*}

Since we can see that
\begin{equation*}
W(s,T)=\sum_{\lambda_1=0}^{\infty}\frac{W^{(\lambda_1)}(1,T)}{\lambda_1!}(s-1)^{\lambda_1},
\end{equation*}
in a similar manner,
\begin{align*}
&\quad
R_3 \\
&=
(-1)^{\mu+\nu}\mu !\nu !x\sum_{\substack{\lambda_1+l=\mu+\nu+1 \\ 0\leq \lambda_1,l}}\frac{W^{(\lambda_1)}(1,T)}{\lambda_1!}\left(\sum_{r=0}^{l}\frac{(-1)^r}{(l-r)!}(\log x)^{l-r} \right) \\
&\quad
+(-1)^{\mu}\mu !x\sum_{\substack{\lambda_1+l+n=\mu+\nu \\ 0\leq \lambda_1,l \\ \mu \leq n}}\frac{W^{(\lambda_1)}(1,T)n!c_n}{(n-\mu)!\lambda_1!}\left(\sum_{r=0}^{l}\frac{(-1)^r}{(l-r)!}(\log x)^{l-r}\right) \\
&\quad
+(-1)^{\nu}\nu !x\sum_{\substack{\lambda_1+l+m=\mu+\nu \\ 0\leq \lambda_1,l \\ \nu \leq m}}\frac{W^{(\lambda_1)}(1,T)m!c_m}{(m-\nu)!\lambda_1!}\left(\sum_{r=0}^{l}\frac{(-1)^r}{(l-r)!}(\log x)^{l-r}\right).
\end{align*}
We note that the order of $R_3$ is at least $x(\log x)^{\mu+\nu+1}$.

From the above computations, we obtain
\begin{align*}
&\quad \sum_{mn\leq x}a_k(m)D_{\mu \nu}(n) \\
&=(-1)^{\mu+\nu+1}\frac{\mu !\nu !}{(\mu+\nu+2)!}(k+1)x(\log x)^{\mu+\nu+2} \\
&\quad+(-1)^{\mu+\nu+1}\mu !\nu !x\sum_{g=1}^{k}\sum_{\lambda=0}^{\mu+\nu+1}\frac{1}{(\mu+\nu+\lambda-l)!}\frac{(\log x)^{\mu+\nu+1-\lambda}}{(z_g-1)^{\lambda+1}} \\
&\quad+\sum_{g=1}^{k}\zeta^{(\mu)}(z_g)\zeta^{(\nu)}(z_g)\frac{x^{z_g}}{z_g}+x\sum_{\lambda=1}^{\mu+\nu+1}C_{\mu,\nu}'(\lambda)(\log x)^{\mu+\nu+2-\lambda} \\
&\quad+x\sum_{\lambda_1=1}^{\mu+\nu}\sum_{g=1}^{k}\sum_{\lambda=0}^{\mu+\nu+1-\lambda_1}C_{\mu,\nu}''(\lambda, \lambda_1)\frac{(\log x)^{\mu+\nu+1-\lambda_1-\lambda}}{(z_g-1)^{\lambda+1}} \\
&\quad+O_{\mu,\nu,k,\varepsilon}((x^c+x^bT^{-1})T^{\varepsilon}(\log x)^{\mu+\nu+1}),
\end{align*}
where $C_{\mu,\nu}'(\lambda)$ and $C_{\mu,\nu}''(\lambda_1)$ are some constants.

This leads to
\begin{align*}
&\quad \sum_{\frac{T}{2\pi}\leq mn\leq \frac{T+U}{2\pi}}a_k(m)D_{\mu \nu}(n)(\log mn)^{2j-\mu-\nu} \\
&=(-1)^{\mu+\nu+1}\frac{\mu !\nu !}{(\mu+\nu+2)!}(k+1)\frac{U}{2\pi}\left(\log \frac{T}{2\pi} \right)^{2j+2} \\
&\quad +(-1)^{\mu+\nu+1}\mu !\nu !\frac{U}{2\pi}\sum_{g=1}^{k}\sum_{\lambda=0}^{\mu+\nu+1}\frac{1}{(\mu+\nu+1-\lambda)!}\frac{(\log \frac{T}{2\pi})^{2j+1-\lambda}}{(z_g-1)^{\lambda+1}} \\
&\quad +(-1)^{\mu+\nu}\mu !\nu !\frac{U}{2\pi}\left(\log \frac{T}{2\pi}\right)^{2j-\mu-\nu}\sum_{g=1}^{k}\frac{\left(\frac{T}{2\pi} \right)^{z_g-1}}{(z_g-1)^{\mu+\nu+2}} \\
&\quad+O_{\mu,\nu,k}\left(U(\log T)^{2j+1} \right).
\end{align*}

To deduce the last main term, we used that
\begin{align*}
(T+U)^{z_g}-T^{z_g}
&=
T^{z_g}\left( \left(1+\frac{U}{T} \right)^{z_g}-1 \right) \\
&=
z_gUT^{z_g-1}+O_k(U^2|T^{z_g-2}|) \\
&=
UT^{z_g-1}+O_k(U(\log T)^{-1}),
\end{align*}

\begin{equation*}
\frac{1}{z_g}=\frac{1}{1-\frac{2}{L}\theta_g+O(L^{-2})}=1+O_k(L^{-1})
\end{equation*}
and
\begin{equation*}
\frac{1}{(z_g-1)^{\lambda}}=\frac{1}{(-2\nu_gL^{-1}+O_k(L^{-2}))^{\lambda}}=(-2\nu_gL^{-1})^{-\lambda}+O_k(L^{\lambda-1})
\end{equation*}
for positive integer $\lambda$. because
\begin{equation*}
z_g=1-\frac{2}{L}\theta_g+O_k(L^{-2}),
\end{equation*}
where $L=\log \frac{T}{2\pi}$.

\subsection*{Step 4}
From the previous steps, recalling (\ref{square}), we obtain
\begin{align*}
I_1
&=(-1)^{j+1}(k+1)\frac{U}{2\pi}\left(\log \frac{T}{2\pi} \right)^{2j+2}\sum_{\mu=0}^{j}\sum_{\nu=0}^{j}\binom{j}{\mu}\binom{j}{\nu}\frac{\mu !\nu !}{(\mu+\nu+2)!}\left(-\frac{1}{2} \right)^{2j-\mu-\nu} \\
&\quad +(-1)^{j+1}\frac{U}{2\pi}\left(\log \frac{T}{2\pi} \right)^{2j+1}\sum_{\mu=0}^{j}\sum_{\nu=0}^{j}\binom{j}{\mu}\binom{j}{\nu}\mu !\nu !\left(-\frac{1}{2}\right)^{2j-\mu-\nu} \\
&\quad \times \sum_{g=1}^{k}\frac{1}{z_g-1}\sum_{\lambda=0}^{\mu+\nu+1}\frac{1}{(\mu+\nu+1-\lambda)!}\frac{(\log \frac{T}{2\pi})^{-\lambda}}{(z_g-1)^{\lambda}} \\
&\quad +(-1)^{j}\frac{U}{2\pi}\sum_{\mu=0}^{j}\sum_{\nu=0}^{j}\binom{j}{\mu}\binom{j}{\nu}\mu !\nu !\left(-\frac{1}{2}\right)^{2j-\mu-\nu}\left(\log \frac{T}{2\pi} \right)^{2j-\mu-\nu}\sum_{g=1}^{k}\frac{\left(\frac{T}{2\pi} \right)^{z_g-1}}{(z_g-1)^{\mu+\nu+2}} \\
&\quad +O_{j,k}\left(U(\log T)^{2j+1} \right) \\
&=(-1)^{j+1}(k+1)\frac{U}{2\pi}\left(\log \frac{T}{2\pi} \right)^{2j+2}\sum_{\mu=0}^{j}\sum_{\nu=0}^{j}\binom{j}{\mu}\binom{j}{\nu}\frac{\mu !\nu !}{(\mu+\nu+2)!}\left(-\frac{1}{2} \right)^{2j-\mu-\nu} \\
&\quad +(-1)^{j+1}\frac{U}{2\pi}\left(\frac{1}{2}\log \frac{T}{2\pi} \right)^{2j+2}\sum_{\mu=0}^{j}\sum_{\nu=0}^{j}\binom{j}{\mu}\binom{j}{\nu}\mu !\nu !\sum_{g=1}^{k}\frac{1}{\theta_g^{\mu+\nu+2}}\sum_{\lambda=0}^{\mu+\nu+1}\frac{(-2\theta_g)^{\lambda}}{\lambda !} \\
&\quad +(-1)^{j}\frac{U}{2\pi}\left(\frac{1}{2}\log \frac{T}{2\pi} \right)^{2j+2}\sum_{\mu=0}^{j}\sum_{\nu=0}^{j}\binom{j}{\mu}\binom{j}{\nu}\mu !\nu !\sum_{g=1}^{k}\frac{\left(\frac{T}{2\pi} \right)^{z_g-1}}{\theta_g^{\mu+\nu+2}} \\
&\quad+O_{j,k}\left(U(\log T)^{2j+1} \right).
\end{align*}
As for the first term,
\begin{align*}
&\quad
\sum_{\mu=0}^{j}\sum_{\nu=0}^{j}\binom{j}{\mu}\binom{j}{\nu}\frac{\mu !\nu !}{(\mu+\nu+2)!}\left(-\frac{1}{2} \right)^{2j-\mu-\nu} \\
&=
(j!)^2\sum_{\mu=0}^{j}\sum_{\nu=0}^{j}\frac{1}{(j-\mu)!(j-\nu)!(\mu+\nu+2)!}\left(-\frac{1}{2} \right)^{2j-\mu-\nu} \\
&=
(j!)^2\sum_{\mu=0}^{j}\sum_{\nu=0}^{j}\frac{1}{\mu !\nu !(2j+2-\mu-\nu)!}\left(-\frac{1}{2} \right)^{\mu+\nu} \\
&=
\frac{(j!)^2}{(2j+2)!}\sum_{\mu=0}^{j}\binom{2j+2}{\mu}\left(-\frac{1}{2}\right)^{\mu}\sum_{\nu=0}^{j}\binom{2j+2-\mu}{\nu}\left(-\frac{1}{2}\right)^{\nu}.
\end{align*}
Here we note that
\begin{align*}
0
&=
\left(1-\frac{1}{2}-\frac{1}{2} \right)^{2j+2}=\sum_{0\leq \mu+\nu\leq 2j+2}\frac{(2j+2)!}{\mu !\nu !(2j+2-\mu-\nu)!}\left(-\frac{1}{2} \right)^{\mu+\nu} \\
&=
\sum_{\mu=0}^{2j+2}\binom{2j+2}{\mu}\left(-\frac{1}{2}\right)^{\mu}\sum_{\nu=0}^{2j+2-\mu}\binom{2j+2-\mu}{\nu}\left(-\frac{1}{2}\right)^{\nu}.
\end{align*}
Thus we have
\begin{align*}
&\quad
\sum_{\mu=0}^{j}\binom{2j+2}{\mu}\left(-\frac{1}{2}\right)^{\mu}\sum_{\nu=0}^{j}\binom{2j+2-\mu}{\nu}\left(-\frac{1}{2}\right)^{\nu} \\
&=
\left(1-\frac{1}{2}-\frac{1}{2} \right)^{2j+2} \\
&\quad
-2\sum_{\mu=j+1}^{2j+2}\binom{2j+2}{\mu}\left(-\frac{1}{2}\right)^{\mu}\sum_{\nu=0}^{2j+2-\mu}\binom{2j+2-\mu}{\nu}\left(-\frac{1}{2}\right)^{\nu} \\
&\quad
+\binom{2j+2}{j+1}\left(-\frac{1}{2}\right)^{2j+2} \\
&=
-2\sum_{\mu=j+1}^{2j+2}\binom{2j+2}{\mu}\left(-\frac{1}{2}\right)^{\mu}\left(\frac{1}{2} \right)^{2j+2-\mu}+\binom{2j+2}{j+1}\left(-\frac{1}{2}\right)^{2j+2} \\
&=
-\frac{1}{2^{2j+2}}\left(2\sum_{\mu=j+1}^{2j+2}\binom{2j+2}{\mu}(-1)^{\mu}-\binom{2j+2}{j+1}\right) \\
&=
\frac{1}{2^{2j+2}}\left(2\sum_{\mu=0}^{j}\binom{2j+2}{\mu}(-1)^{\mu}+\binom{2j+2}{j+1}\right).
\end{align*}
The sum is the coefficient of $x^j$ in $(1-x)^{2j+2}(1-x)^{-1}$.
Thus we can see that
\begin{align*}
&\quad
\sum_{\mu=0}^{j}\sum_{\nu=0}^{j}\binom{j}{\mu}\binom{j}{\nu}\frac{\mu !\nu !}{(\mu+\nu+2)!}\left(-\frac{1}{2} \right)^{2j-\mu-\nu} \\
&=
\frac{(j!)^2}{2^{2j+2}(2j+2)!}\left(2\binom{2j+1}{j}(-1)^j+\binom{2j+2}{j+1}\right) \\
&=
\frac{1+(-1)^j }{2^{2j+2}(j+1)^2}.
\end{align*}

On the second term, putting $u=\mu+\nu+1-\lambda$ and dividing the sum to four parts according as the conditions $u=0,\ 1\leq u \leq j$ with $0\leq i\leq u-1$, $1\leq u \leq j$ with $u\leq i\leq j$ and $j+1\leq u\leq 2j+1$, we have

\begin{align*}
&\quad \sum_{\mu=0}^{j}\sum_{\nu=0}^{j}\binom{j}{\mu}\binom{j}{\nu}\mu !\nu !\sum_{g=1}^{k}\frac{1}{\theta_g^{\mu+\nu+2}}\sum_{\lambda=0}^{\mu+\nu+1}\frac{(-2\theta_g)^{\lambda}}{\lambda !} \\
&=\sum_{g=1}^{k}\sum_{\mu=0}^{j}\sum_{\nu=0}^{j}\sum_{u=0}^{\mu+\nu+1}\binom{j}{\mu}\binom{j}{\nu}\mu !\nu !\frac{1}{\theta_g^{\mu+\nu+2}}\frac{(-2\theta_g)^{\mu+\nu+1-u}}{(\mu+\nu+1-u)!} \\
&=\sum_{g=1}^{k}\frac{1}{\theta_g}\sum_{\mu=0}^{j}\sum_{\nu=0}^{j}\binom{j}{\mu}\binom{j}{\nu}\mu !\nu !\frac{(-2)^{\mu+\nu+1}}{(\mu+\nu+1)!}\\
&\quad+\sum_{g=1}^{k}\sum_{u=1}^{j}\frac{1}{\theta_g^{u+1}}\sum_{\mu=0}^{u-1}\sum_{\nu=u-1-i}^{j}\binom{j}{i}\binom{j}{h}\mu !\nu !\frac{(-2)^{\mu+\nu+1-u}}{(\mu+\nu+1-u)!} \\
&\quad+\sum_{g=1}^{k}\sum_{u=1}^{j}\frac{1}{\theta_g^{u+1}}\sum_{\mu=u}^{j}\sum_{\nu=0}^{j}\binom{j}{\mu}\binom{j}{\nu}\mu !\nu !\frac{(-2)^{\mu+\nu+1-u}}{(\mu+\nu+1-u)!} \\
&\quad+\sum_{g=1}^{k}\sum_{u=j+1}^{2j+1}\frac{1}{\theta_g^{u+1}}\sum_{\mu=u-1-j}^{j}\sum_{\nu=u-1-\mu}^{j}\binom{j}{\mu}\binom{j}{\mu}\mu !\nu !\frac{(-2)^{\mu+\nu+1-u}}{(\mu+\nu+1-u)!} \\
&=S_1+S_2+S_3+S_4,
\end{align*}
say.

To calculate these sums, we prepare a lemma on combinatorics.
\begin{lemma}\label{combi}
For non-negative integers $j$ and $u \ (j\geq u)$,
\begin{align*}
&\quad \sum_{\mu=0}^{j-u}\binom{2j+1-u}{\mu}\sum_{\nu=0}^{j}\binom{2j+1-u-\mu}{\nu}(-2)^{2j+1-u-\mu-\nu}\\
&=(-1)^{j+1}\binom{2j-u}{j}\{ 1+(-1)^{-u}\}.
\end{align*}
\end{lemma}
\begin{proof}
Since
\begin{align*}
&\quad
(1+1-2)^{2j+1-u} \\
&=
\sum_{\mu=0}^{2j+1-u}\binom{2j+1-u}{\mu}\sum_{\nu=0}^{2j+1-u-\mu}\binom{2j+1-u-\mu}{\nu}(-2)^{2j+1-u-\mu-\nu},
\end{align*}
we have
\begin{align*}
&\quad \sum_{\mu=0}^{j-u}\binom{2j+1-u}{\mu}\sum_{\nu=0}^{j}\binom{2j+1-u-\mu}{\nu}(-2)^{2j+1-u-\mu-\nu}\\
&=(1+1-2)^{2j+1-u} \\
&\quad-\sum_{\mu=j-u+1}^{2j+1-u}\binom{2j+1-u}{\mu}\sum_{\nu=0}^{2j+1-u-\mu}\binom{2j+1-u-\mu}{\nu}(-2)^{2j+1-u-\mu-\nu}\\
&\quad-\sum_{\nu=j+1}^{2j+1-u}\binom{2j+1-u}{\nu}\sum_{\mu=0}^{2j+1-u-\nu}\binom{2j+1-u-\nu}{\mu}(-2)^{2j+1-u-\nu-\mu}\\
&=-\sum_{\mu=j-u+1}^{2j+1-u}\binom{2j+1-u}{\mu}(-1)^{2j+1-u-\mu}-\sum_{\nu=j+1}^{2j+1-u}\binom{2j+1-u}{\nu}(-1)^{2j+1-u-\nu}\\
&=-\sum_{\mu=0}^{j}\binom{2j+1-u}{\mu}(-1)^{\mu}-\sum_{\nu=0}^{j-u}\binom{2j+1-u}{\nu}(-1)^{\nu}.
\end{align*}
These sums are coefficients of $x^{j}$ and $x^{j-u}$ in $(1-x)^{2j+1-u}(1-x)^{-1}$ and are therefore equal to the coefficient
of $x^{j}$ and $x^{j-u}$ in $(1-x)^{2j-u}$. Thus we obtain
\begin{align*}
&\quad \sum_{\mu=0}^{j-u}\binom{2j+1-u}{\mu}\sum_{\nu=0}^{j}\binom{2j+1-u-\mu}{\nu}(-2)^{2j+1-u-\mu-\nu}\\
&=(-1)^{j+1}\binom{2j-u}{j}+(-1)^{j+1-u}\binom{2j-u}{j-u} \\
&=(-1)^{j+1}\binom{2j-u}{j}\{1+(-1)^{-u}\}.
\end{align*}
\end{proof}

By Lemma \ref{combi} with $u=0$, when $k\neq 0$,
\begin{align*}
S_1
&=(j!)^2\sum_{g=1}^{k}\frac{1}{\theta_g}\sum_{\mu=0}^{j}\sum_{\nu=0}^{j}\frac{(-2)^{2j+1-i-h}}{\mu !\nu !(2j+1-\mu-\nu)!}\\
&=\frac{(j!)^2}{(2j+1)!}\sum_{g=1}^{k}\frac{1}{\theta_g}\sum_{\mu=0}^{j}\binom{2j+1}{\mu}\sum_{\nu=0}^{j}\binom{2j+1-\mu}{\nu}(-2)^{2j+1-\mu-\nu}\\
&=(-1)^{j+1}2\frac{(j!)^2}{(2j+1)!}\binom{2j}{j}\sum_{g=1}^{k}\frac{1}{\theta_g} \\
&=(-1)^{j+1}\frac{2}{2j+1}\sum_{g=1}^{k}\frac{1}{\theta_g}=(-1)^j\frac{2}{2j+1}.
\end{align*}
At the last equality, we use the fact
\begin{equation*}
\sum_{g=1}^{k}\frac{1}{\theta_g}=-1.
\end{equation*}
This can be obtained by the Newton-Girard formulas.
We note that if $k=0$, then $S_1=0$.

On $S_2$, recalling the proof of Lemma \ref{combi}, we see that
\begin{align*}
S_2
&=\sum_{u=1}^{j}(j!)^2\sum_{g=1}^{k}\frac{1}{\theta_g^{u+1}}\sum_{\mu=0}^{u-1}\frac{1}{(j-\mu)!}\sum_{\nu=0}^{j+\mu+1-u}\frac{(-2)^{\nu}}{\nu !(j+\mu+1-u-\nu)!}\\
&=\sum_{u=1}^{j}\frac{(j!)^2}{(2j+1-u)!}\sum_{g=1}^{k}\frac{1}{\theta_g^{u+1}}\sum_{\mu=0}^{u-1}\binom{2j+1-u}{\mu+j+1-u}\sum_{\nu=0}^{\mu+j+1-u}\binom{\mu+j+1-u}{\nu}(-2)^{\nu}\\
&=\sum_{u=1}^{j}\frac{(j!)^2}{(2j+1-u)!}\sum_{g=1}^{k}\frac{1}{\theta_g^{u+1}}\sum_{\mu=0}^{u-1}\binom{2j+1-u}{\mu+j+1-u}(-1)^{\mu+j+1-u}\\
&=\sum_{u=1}^{j}\frac{(j!)^2}{(2j+1-u)!}\sum_{g=1}^{k}\frac{1}{\theta_g^{u+1}}\sum_{\mu=j+1-u}^{j}\binom{2j+1-u}{\mu}(-1)^{\mu}\\
&=\sum_{u=1}^{j}\frac{(j!)^2}{(2j+1-u)!}\sum_{g=1}^{k}\frac{1}{\theta_g^{u+1}}\left\{(-1)^j\binom{2j-u}{j}-(-1)^{j-u}\binom{2j-u}{j-u} \right\}\\
&=(-1)^j\sum_{u=1}^{j}\frac{j!}{2j+1-u}\frac{1}{(j-u)!}\{ 1-(-1)^{-u}\}\sum_{g=1}^{k}\frac{1}{\theta_g^{u+1}}.
\end{align*}

By Lemma \ref{combi},
\begin{align*}
S_3
&=\sum_{u=1}^{j}(j!)^2\sum_{g=1}^{k}\frac{1}{\theta_g^{u+1}}\sum_{\mu=u}^{j}\sum_{\nu=0}^{j}\frac{1}{(j-\mu)!(j+1-u+\mu-\nu)!}\frac{(-2)^{j+1-u+\mu-\nu}}{\nu !}\\
&=\sum_{u=1}^{j}(j!)^2\sum_{g=1}^{k}\frac{1}{\theta_g^{u+1}}\sum_{\mu=0}^{j-u}\frac{1}{\mu !}\sum_{\nu=0}^{j}\frac{(-2)^{2j+1-u-\mu-\nu}}{\nu !(2j+1-u-\mu-\nu)!}\\
&=\sum_{u=1}^{j}\frac{(j!)^2}{(2j+1-u)!}\sum_{g=1}^{k}\frac{1}{\theta_g^{u+1}}\sum_{\mu=0}^{j-u}\binom{2j+1-u}{\mu}\sum_{\nu=0}^{j}\binom{2j+1-u-\mu}{\nu}(-2)^{2j+1-u-\mu-\nu}\\
&=(-1)^{j+1}\sum_{u=1}^{j}\frac{1}{2j+1-u}\frac{j!}{(j-u)!}\{ 1+(-1)^{-u}\}\sum_{g=1}^{k}\frac{1}{\theta_g^{u+1}}.
\end{align*}

Since
\begin{equation*}
\sum_{\mu=0}^{2j+1-u}\binom{2j+1-u}{\mu}(-1)^{\mu}
=
\begin{cases}
1 & u=2j+1, \\
0 & otherwise,
\end{cases}
\end{equation*}

\begin{align*}
S_4
&=\sum_{u=j+1}^{2j+1}(j!)^2\sum_{g=1}^{k}\frac{1}{\theta_g^{u+1}}\sum_{\mu=u-1-j}^{j}\frac{1}{(j-\mu)!}\sum_{\nu=0}^{j+\mu+1-u}\frac{(-2)^{\nu}}{\nu !(j+\mu+1-u-\nu)!}\\
&=\sum_{u=j+1}^{2j+1}(j!)^2\sum_{g=1}^{k}\frac{1}{\theta_g^{u+1}}\sum_{\mu=u-1-j}^{j}\frac{1}{(j-\mu)!(j+\mu+1-u)!}\sum_{\nu=0}^{j+\mu+1-u}\binom{j+\mu+1-u}{\nu}(-2)^{\nu}\\
&=\sum_{u=j+1}^{2j+1}\frac{(j!)^2}{(2j+1-u)!}\sum_{g=1}^{k}\frac{1}{\theta_g^{u+1}}\sum_{\mu=0}^{2j+1-u}\binom{2j+1-u}{\mu}(-1)^{\mu}=\sum_{g=1}^{k}\frac{(j!)^2}{\theta_g^{2j+2}}.
\end{align*}

Thus we have
\begin{align*}
&\quad
S_1+S_2+S_3+S_4\\
&=(-1)^j\frac{2}{2j+1}+(-1)^{j+1}2\sum_{u=1}^{j}\frac{1}{2j+1-u}\frac{j!}{(j-u)!}(-1)^{-u}\sum_{g=1}^{k}\frac{1}{\theta_g^{u+1}}+\sum_{g=1}^{k}\frac{(j!)^2}{\theta_g^{2j+2}}.
\end{align*}

On the third term, we see that
\begin{align*}
\sum_{\mu=0}^{j}\sum_{\nu=0}^{j}\binom{j}{\mu}\binom{j}{\nu}\mu !\nu !\sum_{g=1}^{k}\frac{\left(\frac{T}{2\pi} \right)^{z_g-1}}{\theta_g^{\mu+\nu+2}}=(j!)^2\sum_{g=1}^{k}\frac{\left(\frac{T}{2\pi}\right)^{z_g-1}}{\theta_g^{2j+2}} \left(\sum_{\mu=0}^{j}\frac{\theta_g^{\mu}}{\mu !} \right)^2.
\end{align*}

Therefore, when $k\neq 0$,
\begin{align*}
I_1
&=-\frac{(k+1)\{1+(-1)^j \}}{2^{2j+2}(j+1)^2}\frac{U}{2\pi}\left(\log \frac{T}{2\pi} \right)^{2j+2} \\
&\quad -\frac{U}{2^{2j+2}(2j+1)\pi}\left(\log \frac{T}{2\pi} \right)^{2j+2} \\
&\quad +\sum_{u=1}^{j}\frac{1}{2j+1-u}\frac{j!}{(j-u)!}(-1)^{-u}\sum_{g=1}^{k}\frac{1}{\theta_g^{u+1}}\frac{U}{2^{2j+2}\pi}\left(\log \frac{T}{2\pi} \right)^{2j+2} \\
&\quad -(-1)^{j}\sum_{g=1}^{k}\frac{(j!)^2}{\theta_g^{2j+2}}\frac{U}{2^{2j+3}\pi}\left(\log \frac{T}{2\pi} \right)^{2j+2} \\
&\quad +(-1)^{j}(j!)^2\sum_{g=1}^{k}\frac{\left(\frac{T}{2\pi}\right)^{z_g-1}}{\theta_g^{2j+2}} \left(\sum_{\mu=0}^{j}\frac{\theta_g^{\mu}}{\mu !} \right)^2\frac{U}{2^{2j+3}\pi}\left(\log \frac{T}{2\pi} \right)^{2j+2} \\
&\quad+O_{j,k}\left(U(\log T)^{2j+1} \right).
\end{align*}
If $k=0$, then these main terms vanish except for the first.

\subsection*{Step 5}
Finally, we obtain
\begin{align*}
&\qquad
M(T+U)-M(T)\\
&=
\delta_{0,k}\frac{U}{2^{2j+1}(2j+1)\pi}\left(\log \frac{T}{2\pi}\right)^{2j+2} \\
&\quad-\frac{(k+1)\{1+(-1)^j \}}{2^{2j+1}(j+1)^2}\frac{U}{2\pi}\left(\log \frac{T}{2\pi} \right)^{2j+2} \\
&\quad +\sum_{u=1}^{j}\frac{1}{2j+1-u}\frac{j!}{(j-u)!}(-1)^{-u}\sum_{g=1}^{k}\frac{1}{\theta_g^{u+1}}\frac{U}{2^{2j+1}\pi}\left(\log \frac{T}{2\pi} \right)^{2j+2} \\
&\quad +(-1)^{j+1}\sum_{g=1}^{k}\frac{(j!)^2}{\theta_g^{2j+2}}\frac{U}{2^{2j+2}\pi}\left(\log \frac{T}{2\pi} \right)^{2j+2} \\
&\quad +(-1)^{j}(j!)^2\sum_{g=1}^{k}\frac{\left(\frac{T}{2\pi}\right)^{z_g-1}}{\theta_g^{2j+2}} \left(\sum_{\mu=0}^{j}\frac{\theta_g^{\mu}}{\mu !} \right)^2\frac{U}{2^{2j+2}\pi}\left(\log \frac{T}{2\pi} \right)^{2j+2} \\
&\quad+O_{j,k}\left(U(\log T)^{2j+1}\right)
\end{align*}
This completes the proof for the special $T$ which are chosen at the beginning of the proof.

To complete the proof, we take away the condition on $T$.
When $T$ increases continuously in bounded interval, the number of relevant $|Z^{(j)}(\gamma_k)|^2$ is at most $O_k(\log T)$ and the order is $O_j(T^{\varepsilon})$.
Thus it is smaller than the error in our main theorem that the contribution of these terms.
Thus the formula is true for all $T>T_0$.

\section*{Acknowledgement}
I would like to thank my supervisor Professor Kohji Matsumoto for useful advice.
I am grateful to the seminar members for some helpful remarks and discussions.

\end{document}